\documentclass[11pt,a4paper]{article}
\usepackage{amsmath, amsfonts, amsthm,color, latexsym, amssymb}
\newtheorem{theorem}{Theorem}[section]
\newtheorem{proposition}[theorem]{Proposition}
\newtheorem{corollary}[theorem]{Corollary}
\newtheorem{example}[theorem]{Example}

\newtheorem{remark}[theorem]{Remark}

\newtheorem{lemma}[theorem]{Lemma}
\newtheorem{definition}[theorem]{Definition}
\usepackage{amscd}
\begin{document}
\title{\sc Uniform approximation on ideals of multilinear mappings}

\author{Geraldo Botelho\thanks{Supported by CNPq Project 202162/2006-0.}~, Pablo Galindo\thanks{ Supported
partially by MEC-FEDER Project  MTM 2007-64521.} ~and Leonardo
Pellegrini\thanks{Supported by ProIP Project-USP.\hfill \newline
}}

\date{}
\maketitle \vspace*{-1.0em}

\begin{abstract}
For each ideal of multilinear mappings $\cal M$ we explicitly construct a corresponding ideal $^{a}\!{\cal
M}$ such that multilinear forms in $^{a}\!{\cal M}$ are exactly those which can be approximated, in the
uniform norm, by multilinear forms in ${\cal M}$. This construction is then applied to finite type, compact,
weakly compact and absolutely summing multilinear mappings. It is also proved that the correspondence ${\cal
M} \mapsto \,\! ^{a}\!{\cal M}$ is Aron-Berner stability preserving.
\end{abstract}

\section*{Introduction}

The theory of ideals of multilinear mappings (multi-ideals)
between Banach spaces (see \cite{note, studia, bpr, braunss, bj,
cdps,FloretResults, Floret, fh, FloretGarcia, pg, Pietsch} and
references therein) studies, as in the theory of linear operator
ideals, plenty of non-closed ideals. So the question of describing
their closures in the uniform norm is quite natural. Our main aim
is to construct an approximation scheme for multi-ideals similar
to the H. Jarchow and A. Pe{\l}czy\'nski description of the closed
injective hull of an operator ideal, which can be found in
\cite[Section 20.7]{jarchow}.

 Given a multi-ideal $\cal M$,  we construct a corresponding
multi-ideal $^{a}\!{\cal M}$ such that every multilinear form in
$^{a}\!{\cal M}$ can be approximated, in the uniform norm, by
multilinear forms in $\cal M$. For  multilinear mappings taking
values in a Banach space $F$ the approximation takes place, like
in the linear case, in the larger space $\ell_\infty(B_{F^*})$. We
do not only  prove the existence of such multi-ideal $^{a}\!{\cal
M}$, we give it an explicit description that roughly speaking
means that a multilinear mapping $A$ can be approximated by
elements in $\cal M$ if the norm of any sum of images of $A$ is
almost "dominated" by the norm of the sum of the corresponding
images by some element in $\cal M,$ (see Definition 2.1). Relying
on such description we obtain several properties of
$^{a}\!{\cal M}.$ This is done in section 2.\\
\indent It should be noted that, as usual, there are several {\it
a priori} possibilities to transpose the Jarchow-Pe{\l}czy\'nski
construction to the multilinear case. The direct and most obvious
transposition simply does not work. So an important step was to
identify, among all possible multilinear generalizations of the
Jarchow-Pe{\l}czy\'nski construction, the one that performs the
desired approximation. The results we prove show that we have
selected the correct
definition. \\
\indent The rest of paper is organized as follows. Section 1 provides the basic notions and
fixes the notation.
 In section 3 we apply the approximation theorem to
finite type and compact mappings. Denoting the multi-ideal of
multilinear mappings of finite type by ${\cal L}_{f}$, we prove
that its corresponding ideal $^{a}\!({\cal L}_{f})$ coincides,
modulo the approximation property, with the extensively studied
(see  \cite{ahv} or \cite{din} for instance) class of multilinear
mappings which are weakly continuous on bounded sets. The case of
weakly compact and absolutely summing mappings is studied in
section 4, where we prove multilinear counterparts of some
important linear results. In section 5 we show that the
correspondence ${\cal M} \mapsto\,\! ^{a}\!{\cal M}$ preserves the
stability of ${\cal M}$ with respect to Aron-Berner extensions of
multilinear mappings to the bidual spaces; this will extend to the
multilinear setting known results about bitranspose linear
operators. A contribution to the linear theory is also obtained.

\section{Background and notation}
Throughout $n$ is a positive integer, $E, E_1, \ldots, E_n, F,
G_1, \ldots, G_n$ and $H$ are (real or complex) Banach spaces.
${\cal L}(E;F)$ denotes the Banach space, endowed with the usual
sup norm, of bounded linear operators from $E$ to $F$ and ${\cal
L}(E_1, \ldots, E_n;F)$ the Banach space, endowed with the usual
sup norm, of continuous $n$-linear mappings from $E_1 \times
\cdots \times E_n$ to $F$. If $F$ is the scalar field we simply
write $E^*$ and ${\cal L}(E_1, \ldots, E_n)$. If $E_1 = \cdots =
E_n = E$ we write ${\cal L}(^nE;F)$ and ${\cal L}(^nE)$. Linear
combinations of mappings of the form $A(x_1, \ldots, x_n) =
\varphi_1(x_1) \cdots \varphi_n(x_n) b$, where $\varphi_j \in
E_j^*$ and $b \in F$, are called {\it $n$-linear mappings of
finite type}. The space of all such mappings is denoted by ${\cal
L}_f(E_1, \ldots, E_n;F)$. The mappings belonging to its closure
$\overline{{\cal L}_f}$ are called {\it approximable}. For each
$A\in {\cal L}(E_1, \ldots, E_n;F),$ we denote by $A_L \in {\cal
L}(E_1 \widehat\otimes_\pi \cdots \widehat\otimes_\pi E_n;F)$ its
linearization on the
completed $n$-fold projective tensor product 
 that is defined by
$$A_L(x_1 \otimes \cdots \otimes x_n) = A(x_1,
\ldots, x_n){\rm ~for~all~}x_j \in E_j.
$$
For the general theory of multilinear mappings between Banach
spaces we refer to S. Dineen \cite{din}
\begin{definition}[Ideals of multilinear mappings or multi-ideals]
\label{Definition 1.3} $~${\rm An {\it ideal of $n$-linear mappings} (or {\it $n$-ideal}) $\cal M$ is a
subclass of the class of all continuous
 $n$-linear mappings between Banach spaces such that for Banach spaces $E_1, \ldots, E_n$ and $F$, the components ${\cal M}
 (E_1, \ldots, E_n;F) := {\cal L}(E_1, \ldots, E_n;F) \cap {\cal M}$ satisfy:
 \\
\noindent (i) ${\cal M}(E_1, \ldots, E_n;F)$ is a linear subspace of
${\cal L}(E_1, \ldots, E_n;F)$ which contains the $n$-linear mappings of finite type.\\
(ii) The ideal property: if $A \in {\cal M}(E_1, \ldots, E_n;F)$,
$u_j \in {\cal L}(G_j;E_j)$ for $j = 1, \ldots, n$ and $t \in
{\cal L}(F;H)$, then $t \circ A \circ (u_1, \ldots, u_n)$ is in
${\cal M}(G_1, \ldots, G_n;H)$.

\medskip

\noindent If there is a function $\|\cdot\|_{\cal M} \colon {\cal
M} \longrightarrow \mathbb{R}^+$ satisfying

\medskip

\noindent (i') There is $0 < p \leq 1$ such that $\|\cdot\|_{\cal M}$ restricted to ${\cal M}(E_1, \ldots,
E_n;F)$ is a $p$-norm for all
 Banach spaces $E_1, \ldots, E_n$ and $F$,\\
\noindent (ii') $\|A \colon \mathbb{K}^n \longrightarrow \mathbb{K} : A(\lambda_1,\ldots,
 \lambda_n)
 = \lambda_1 \cdots \lambda_n \|_{\cal M} = 1$ for all $n$,\\
\noindent (iii') If $A \in {\cal M}(E_1, \ldots, E_n;F)$, $u_j \in
{\cal L}(G_j;E_j)$ for $j = 1, \ldots, n$ and $t \in {\cal
L}(F;H)$, then $\|t \circ A \circ (u_1, \ldots, u_n)\|_{\cal M}
\leq \|t\| \|A\|_{\cal M} \|u_1\|\cdots \|u_n\|$,

\smallskip

\noindent then ${\cal M}$ is called a {\it quasi-normed} ({\it normed} if $p = 1$) {\it $n$-ideal}. {\it
Quasi-Banach} ({\it Banach} if $p = 1$) {\it $n$-ideals} are defined in the obvious way. A {\it (Banach)
multi-ideal} is a sequence $({\cal M}_n)_{n=1}^\infty$ where each ${\cal M}_n$ is a (Banach) $n$-ideal. }
\end{definition}
When $n=1$ we recover the classical theory of operator ideals, for
which the reader is referred to \cite{df}.
\medskip

\indent A Banach $n$-ideal ${\cal M}$ is said to be {\it closed} if each component ${\cal M}(E_1, \ldots,
E_n;F)$ is a (sup-norm) closed subspace of ${\cal L}(E_1, \ldots, E_n;F)$. A Banach multi-ideal $({\cal
M}_n)_{n=1}^\infty$ is
closed if each ${\cal M}_n$ is closed. \\
\indent An $n$-ideal ${\cal M}$ is {\it injective} if whenever $A
\in {\cal L}(E_1, \ldots, E_n;F)$, $i \colon F \longrightarrow G$
is an isometric embedding and $i \circ A \in {\cal M}(E_1, \ldots,
E_n;G)$, then $A \in {\cal M}(E_1, \ldots, E_n;F)$.

\section{The approximation scheme}

\begin{definition}\rm Let $\cal M$ be an  $n$-ideal.
A mapping $A \in {\cal L}(E_1, \ldots, E_n;F)$ is said to be ${\cal M}$-{\it approximable}, in symbols $A \in
{^a}\!{\cal M}(E_1, \ldots, E_n;F)$, if there are a Banach space $G$ and an $n$-linear mapping $B \in {\cal
M}(E_1, \ldots, E_n;G)$ such that for every $\varepsilon >0$ there is $K_{\varepsilon} > 0$ such that
$$\left \|\sum_{i=1}^k A(x_i^1, \ldots, x_i^n)\right \| \leq K_{\varepsilon}
\left \|\sum_{i=1}^k B(x_i^1, \ldots, x_i^n)\right \| +
\varepsilon \sum_{i=1}^k \|x_i^1\|\cdots \|x_i^n\|,$$ for every $k
\in \mathbb{N}$ and any $x_i^j \in E_j$, $j =1, \ldots, n$, $i =
1, \ldots, k$.
\end{definition}

\begin{remark}\rm (a) It is obvious from the definition that, for every $n$-ideal
${\cal M}$, ${^a}\!{\cal M}$ is injective. It is also obvious that ${^a}\!{\cal M} \subseteq {^a}\!{\cal N}$
if ${\cal M} \subseteq {\cal N}$.
\\(b) Let us see that ${^a}\!({^a}\!{\cal
M}) = {^a}\!{\cal M}$: it is clear that $^a\!{\cal M}$ contains
$\cal M$, so ${^a}\!{\cal M} \subseteq {^a}\!({^a}\!{\cal M})$.
For the converse, given $A \in {^a}\!({^a}\!{\cal M})(E_1, \ldots,
E_n;F)$, there are a Banach space $G$ and an $n$-linear mapping $B
\in {^a}\!{\cal M}(E_1, \ldots, E_n;G)$ such that for every
$\varepsilon
>0$ there is $K_{\varepsilon}$ such that
$$\left \|\sum_{i=1}^k A(x_i^1, \ldots, x_i^n)\right \| \leq K_{\varepsilon}
\left \|\sum_{i=1}^k B(x_i^1, \ldots, x_i^n)\right \| +
\varepsilon \sum_{i=1}^k \|x_i^1\|\cdots \|x_i^n\|,$$ for every $k
\in \mathbb{N}$ and any $x_i^j \in E_j$, $j =1, \ldots, n$, $i =
1, \ldots, k$. Since $B \in {^a}\!{\cal M}$, there are a Banach
space $G$ and an $n$-linear mapping $B \in {^a}\!{\cal M}(E_1,
\ldots, E_n;G)$ such that for every $\varepsilon
>0$ there is $N_{\varepsilon}$ such that
$$\left \|\sum_{i=1}^k B(x_i^1, \ldots, x_i^n)\right \| \leq N_{\varepsilon}
\left \|\sum_{i=1}^k C(x_i^1, \ldots, x_i^n)\right \| +
\varepsilon \sum_{i=1}^k \|x_i^1\|\cdots \|x_i^n\|,$$ for every $k
\in \mathbb{N}$ and any $x_i^j \in E_j$, $j =1, \ldots, n$, $i =
1, \ldots, k$. Given $\varepsilon > 0$, one easily checks that for
every $k \in \mathbb{N}$ and any $x_i^j \in E_j$, $j =1, \ldots,
n$, $i = 1, \ldots, k,$
\begin{eqnarray*}
\left \|\sum_{i=1}^k A(x_i^1, \ldots, x_i^n)\right \| \!\! & \leq
&\!\! L_{\varepsilon} \left \|\sum_{i=1}^k C(x_i^1,
\ldots, x_i^n)\right \| + \varepsilon \sum_{i=1}^k \|x_i^1\|\cdots
\|x_i^n\|,
\end{eqnarray*}
where $L_{\varepsilon}$ may be chosen as
$K_{\frac{\varepsilon}{2}}\cdot
N_{\frac{\varepsilon}{2K_{\frac{\varepsilon}{2}}}}.$ Thus, $A \in
{^a}\!{\cal M}(E_1,
\ldots, E_n;F)$.\\
(c) For operator ideals, this is the Jarchow-Pe{\l}czy\'nski description of the closed injective hull of a
given operator ideal (see \cite[Theorem 20.7.3]{jarchow}). In particular, when $n = 1$ and ${\cal M}$ is the
ideal $\Pi_p$ of absolutely $p$-summing operators, $1 \leq p < +\infty$, we get ${^a}\!{\cal M} = {\cal H}$,
where ${\cal H}$ is the ideal of absolutely continuous linear
operators (see \cite[Chapter 15]{djt})
\end{remark}

\begin{proposition}\label{proposition} If $\cal M$ is a Banach $n$-ideal, then $^a\!{\cal M}$ is
a closed injective $n$-ideal containing $\cal M$.
\end{proposition}

\begin{proof} Since $^a\!{\cal M}$ contains $\cal M$, it also contains
the $n$-linear mappings of finite type.

Given $A_1, A_2 \in {^a}\!{\cal M}(E_1, \ldots, E_n;F)$, let $B_j
\in {\cal M}(E_1, \ldots, E_n;G_j)$, $j = 1,2$, for which the
definition holds. Define $G := G_1 \oplus_1 G_2$ and $B \in {\cal
L}(E_1, \ldots, E_n;G)$ by $B(x_1, \ldots, x_n) = (B_1(x_1,
\ldots, x_n), B_2(x_1, \ldots, x_n)).$ $B$ belongs to $\cal M$ as
$B = i_i \circ B_1 + i_2 \circ B_2$, where $i_j$ are the canonical
inclusions. It is a routine computation to verifiy that such $B$
fulfills
the conditions of the definition for $A_1 + A_2.$  
Thus $A_1 + A_2 \in {^a}\!{\cal M}(E_1, \ldots, E_n;F)$.\\
\indent The ideal property is easily checked.\\
\indent Let $(A_j)_{j=1}^\infty \subseteq {^a}\!{\cal M}(E_1, \ldots, E_n;F)$, $A_j \longrightarrow A \in
{\cal L}(E_1, \ldots, E_n;F)$ in norm. For each $j \in \mathbb{N}$, take $G_j$ and $0 \neq B_j \in {\cal
M}(E_1, \ldots, E_n;G_j)$ associated to $A_j$ according to the definition. Define $G := \left (
\displaystyle\oplus_{j=1}^\infty G_j\right)_1$ and for each $j$ consider the canonical inclusion $i_j \colon
G_j \longrightarrow G$. It is clear each $i_j \circ B_j$ belongs to $\cal M$. Since
\[\sum_{j=1}^\infty \frac{1}{2^j\|B_j\|_{\cal M}} \|i_j
\circ B_j\|_{\cal M} \leq \sum_{j=1}^\infty
\frac{1}{2^j\|B_j\|_{\cal M}} \|i_j\|\| B_j\|_{\cal M} =
\sum_{j=1}^\infty \frac{1}{2^j} < +\infty,\] and $({\cal M}(E_1,
\ldots, E_n;G), \|\cdot\|_{\cal M})$ is a Banach space, the series
$\sum_{j=1}^\infty \frac{1}{2^j\|B_j\|_{\cal M}} i_j \circ B_j$
converges in this space. Call
$$B
:= \sum_{j=1}^\infty \frac{1}{2^j\|B_j\|_{\cal M}} i_j \circ B_j
\in {\cal M}(E_1, \ldots, E_n;G).$$ From $\|\cdot \| \leq
\|\cdot\|_{\cal M}$ (see \cite[Satz 2.2.5]{braunss}) it follows
that such series is pointwise convergent.
Given $\varepsilon >0$, let $j_0 \in \mathbb{N}$ be such that
$\|A_{j_0} - A\|< \displaystyle\frac{\varepsilon}{2}$. Take $K$
such that for every $k \in \mathbb{N}$  and any $x_i^j \in E_j$,
$j =1, \ldots, n$, $i = 1, \ldots, k,$
$$\left \|\sum_{i=1}^k A_{j_0}(x_i^1, \ldots, x_i^n)\right \| \leq K
\left \|\sum_{i=1}^k B_{j_0}(x_i^1, \ldots, x_i^n)\right \| +
\frac{\varepsilon}{2} \sum_{i=1}^k \|x_i^1\|\cdots \|x_i^n\|.$$
 Then, for every $k \in \mathbb{N}$ and any $x_i^j \in E_j$, $j =1,
\ldots, n$, $i = 1, \ldots, k,$
\begin{eqnarray*} \left \|\sum_{i=1}^k A(x_i^1,
\ldots, x_i^n) \right\| \!\!& \leq &\!\! \left \|\sum_{i=1}^k (A -
A_{j_0})(x_i^1, \ldots, x_i^n) \right\| + \left \|\sum_{i=1}^k
A_{j_0}(x_i^1, \ldots, x_i^n) \right\|\\\!\!& \leq &\!\! K \left
\|\sum_{i=1}^k B_{j_0}(x_i^1, \ldots, x_i^n)\right \| +
\varepsilon \sum_{i=1}^k \|x_i^1\|\cdots
\|x_i^n\|\end{eqnarray*} \begin{eqnarray*} \\
\!\!& = &\!\! K 2^{j_0}\|B_{j_0}\|_{\cal M}\left \|\sum_{i=1}^k
\frac{B_{j_0}(x_i^1, \ldots, x_i^n)}{2^{j_0}\|B_{j_0}\|_{\cal
M}}\right \| + \varepsilon \sum_{i=1}^k \|x_i^1\|\cdots
\|x_i^n\|\\
\!\!& \leq &\!\! K 2^{j_0}\|B_{j_0}\|_{\cal M}
\sum_{j=1}^\infty\left \| \sum_{i=1}^k \frac{B_{j}(x_i^1, \ldots,
x_i^n)}{2^{j}\|B_{j}\|_{\cal M}}\right \| +
\varepsilon\sum_{i=1}^k \|x_i^1\|\cdots \|x_i^n\|\\
\!\!& \leq &\!\! K 2^{j_0}\|B_{j_0}\|_{\cal M} \left \|
\sum_{i=1}^k B(x_i^1, \ldots, x_i^n)\right \| +
\varepsilon\sum_{i=1}^k \|x_i^1\|\cdots \|x_i^n\|.
\end{eqnarray*}
 It results that $A$ belongs to
${^a}\!{\cal M}$, completing the proof.
\end{proof}






For every Banach space $E$, by $i_E$ we mean the canonical isometric embedding $E \longrightarrow
\ell_\infty(B_{E^*})$. Besides of performing the desired approximation scheme, next result shows that
${^a}\!{\cal M}$ could have been defined by a condition which seems to be less demanding at first glance.

\begin{theorem}\label{teorema} Let $\cal M$ be an  $n$-ideal. The following
are equivalent for an $n$-linear mapping $A \in {\cal L}(E_1,
\ldots, E_n;F)$:\\
{\rm (a)} $A \in {^a}\!{\cal M}(E_1, \ldots, E_n;F)$.\\
{\rm (b)} For every $\varepsilon > 0$ there is an $n$-linear
mapping $C \in {\cal M}(E_1, \ldots, E_n;\ell_\infty(B_{F^*}))$
such that $\|i_F \circ A - C\| < \varepsilon$.\\
{\rm (c)} There is a set $\Gamma$ and an isometric embedding
$i_F^\Gamma \colon F \longrightarrow \ell_\infty(\Gamma)$ such
that for every $\varepsilon
> 0$ there is an $n$-linear mapping $C \in {\cal M}(E_1, \ldots,
E_n;\ell_\infty(\Gamma))$ such that $\|\i_F^\Gamma \circ A - C\| <
\varepsilon$. \\
{\rm (d)} For every $\varepsilon > 0$ there are a Banach space
$G_\varepsilon$ and an $n$-linear mapping $B_\varepsilon \in {\cal
M}(E_1, \ldots, E_n;G_\varepsilon)$ such that
$$\left \|\sum_{i=1}^k A(x_i^1, \ldots, x_i^n)\right \| \leq
\left \|\sum_{i=1}^k B_\varepsilon(x_i^1, \ldots, x_i^n)\right \|
+ \varepsilon \sum_{i=1}^k \|x_i^1\|\cdots \|x_i^n\|,$$ for every
$k \in \mathbb{N}$ and any $x_i^j \in E_j$, $j =1, \ldots, n$, $i
= 1, \ldots, k$.
\end{theorem}

\begin{proof} (d) $\Longrightarrow$ (b) Let $\varepsilon > 0$ be given.
By assumption, there are $G_\varepsilon$ and $B_\varepsilon \in
{\cal M}(E_1, \ldots, E_n;G_\varepsilon)$ such that for every $k
\in \mathbb{N}$ and any $x_i^j \in E_j$, $j =1, \ldots, n$, $i =
1, \ldots, k,$
$$\left \|\sum_{i=1}^k A(x_i^1, \ldots, x_i^n)\right \| \leq
\left \|\sum_{i=1}^k B_\varepsilon(x_i^1, \ldots, x_i^n)\right \|
+ \varepsilon \sum_{i=1}^k \|x_i^1\|\cdots \|x_i^n\|.$$  
So we have for their respective linearizations $A_L$ and $ (B_\varepsilon)_L,$
$$\left \|A_L \left( \sum_{i=1}^k x_i^1 \otimes \cdots \otimes x_i^n
\right)\right \| \leq \left \|(B_\varepsilon)_L \left(\sum_{i=1}^k x_i^1 \otimes \cdots \otimes x_i^n
\right)\right \| + \varepsilon \sum_{i=1}^k \|x_i^1\|\cdots \|x_i^n\|,$$ for every $\sum_{i=1}^k x_i^1
\otimes \cdots \otimes x_i^n \in E_1 \otimes_\pi \cdots \otimes_\pi E_n$. Taking the infimum over all
representations of a tensor $\theta \in E_1 \otimes_\pi \cdots \otimes_\pi E_n$ in the form $\theta =
\sum_{i=1}^k x_i^1 \otimes \cdots \otimes x_i^n$ it results that $\|A_L(\theta) \| \leq
\|(B_\varepsilon)_L(\theta)\| + \varepsilon\pi(\theta)$ for every $\theta \in E_1 \otimes_\pi \cdots
\otimes_\pi E_n$. Since $A_L$ and $(B_\varepsilon)_L$ are continuous, this inequality also holds on $E_1
\widehat\otimes_\pi \cdots \widehat\otimes_\pi E_n$. Consider the following linear operator:
$$a_\varepsilon \colon E_1 \otimes_\pi \cdots
\otimes_\pi E_n \longrightarrow (G_\varepsilon \oplus (E_1 \otimes_\pi \cdots \otimes_\pi
E_n))_1~,~a_\varepsilon(\theta) = ( (B_\varepsilon)_L(\theta), \varepsilon \theta).$$ The fact that
$a_\varepsilon$ is injective allows us to define a linear operator $b_\varepsilon \colon {\rm
Range}(a_\varepsilon) \longrightarrow F$ by $b_\varepsilon(a_\varepsilon(\theta)) = A_L(\theta)$. From
$$\|b_\varepsilon(a_\varepsilon(\theta)\| = \|A_L(\theta)\| \leq \|(B_\varepsilon)_L(\theta)\| +
\varepsilon\pi(\theta) = \|a_\varepsilon(\theta)\|,$$ we find that
$\|b_\varepsilon\|\leq 1.$ As  a continuous linear operator from
the normed space ${\rm Range}(a_\varepsilon)$ into the injective
Banach space $\ell_\infty(B_{F^*})$, $i_F \circ b_\varepsilon$ can
be extended to a continuous linear operator $\overline{i_F \circ
b_\varepsilon} \colon (G_\varepsilon \oplus (E_1 \otimes_\pi
\cdots \otimes_\pi E_n))_1 \longrightarrow \ell_\infty(B_{F^*})$,
$\|\overline{i_F \circ b_\varepsilon}\| \leq 1$. Define $C \in
{\cal L}(E_1, \ldots, E_n;\ell_\infty(B_{F^*}))$ by
$$C(x_1, \ldots, x_n) = \overline{i_F \circ b_\varepsilon}(
B_\varepsilon(x_1, \ldots, x_n),0).$$ If $j_\varepsilon \colon
G_\varepsilon \longrightarrow (G_\varepsilon \oplus (E_1
\otimes_\pi \cdots \otimes_\pi E_n))_1$ is given by
$j_\varepsilon(y) = (y,0)$, then \linebreak $C = \overline{i_F
\circ b_\varepsilon} \circ j_\varepsilon \circ B_\varepsilon$,
hence $C \in {\cal M}(E_1, \ldots, E_n;\ell_\infty(B_{F^*}))$. For
$(x_1, \ldots, x_n) \in E_1 \times \cdots \times E_n$,
\begin{eqnarray*}
i_F \circ A_L(x_1 \otimes \cdots \otimes x_n) &=&
i_F(b_\varepsilon (a_\varepsilon(x_1 \otimes \cdots \otimes
x_n)))\\
& =& \overline{i_F \circ b_\varepsilon}(((B_\varepsilon)_L(x_1
\otimes \cdots \otimes x_n), \varepsilon x_1 \otimes \cdots
\otimes x_n))\\
& =& C(x_1, \ldots, x_n) + \overline{i_F \circ b_\varepsilon}((0,
\varepsilon x_1 \otimes \cdots \otimes x_n)).
\end{eqnarray*}
\begin{eqnarray*}\text{ Thus, }\;\;
\|i_F \circ A(x_1,  \ldots , x_n) - C(x_1,  \ldots , x_n)\| & = &
\|\overline{i_F \circ b_\varepsilon}((0, \varepsilon x_1 \otimes
\cdots \otimes x_n))\|\\
& \leq & \|\overline{i_F \circ b_\varepsilon}\|\varepsilon \pi(x_1
\otimes \cdots \otimes x_n)\\
& \leq & \varepsilon \|x_1\| \cdots \|x_n\|,
\end{eqnarray*}
which proves that $\|i_F \circ A - C\| \leq \varepsilon$.\\
(c) $\Longrightarrow$ (a) The assumption assures that $i_F^\Gamma
\circ A $ belongs to the sup-norm closure of ${{\cal M}(E_1,
\ldots, E_n;\ell_\infty(\Gamma))}$. It follows easily from the
definition of ${^a}\!{\cal M}$ that $\overline{\cal M} \subseteq
{^a}\!{\cal M}$, therefore $i_F^\Gamma \circ A \in {^a}\!{\cal
M}(E_1, \ldots, E_n;\ell_\infty(\Gamma))$. The injectivity of
${^a}\!{\cal M}$
gives $A \in {^a}\!{\cal M}(E_1, \ldots, E_n;F))$.\\
\indent The proof is complete as (a) $\Longrightarrow$ (d) and (b)
$\Longrightarrow$ (c) are obvious.
\end{proof}

\begin{remark}\rm Given an $n$-ideal $\cal M$,  it follows immediately from Theorem
\ref{teorema} that, for every $E_1, \ldots, E_n$ and $F$,
\begin{eqnarray*}
{^a}\!{\cal M}(E_1, \ldots, E_n;F) \!\!\!& = \!\!\!&\{A \in {\cal
L}(E_1, \ldots, E_n;F) : i_F \circ A \in  \overline{{\cal M}(E_1,
\ldots,
E_n;\ell_\infty(B_{F^*})}\}\\
\!\!\!& = \!\!\!&\{A \in {\cal L}(E_1, \ldots, E_n;F) : i_F^\Gamma
\circ A \in  \overline{{\cal M}(E_1, \ldots, E_n;\ell_\infty(\Gamma))}{\rm ~for} \\
\!\!\!& & \,\,{\rm some~}\Gamma {\rm ~and~some~isometric~
embedding~}i_F^\Gamma \colon F \longrightarrow
\ell_\infty(\Gamma)\}.
\end{eqnarray*}
\end{remark}

In the proof of Theorem \ref{teorema}, if $F$ is injective there
is no need to go to the larger space $\ell_\infty(B_{F^*})$.
\begin{corollary}\label{corollary} Let $\cal M$ be
 an $n$-ideal. If $F$ is an injective Banach space, then ${^a}\!{\cal M}(E_1,
\ldots, E_n;F) = \overline{{\cal M}(E_1, \ldots, E_n;F)}$ for every $E_1, \ldots, E_n$. In particular,
${^a}\!{\cal M}(E_1, \ldots, E_n) = \overline{{\cal M}(E_1, \ldots, E_n)}$ for every $E_1, \ldots, E_n$.
\end{corollary}

\begin{proposition}\label{envoltoria} Let $\cal M$ be a Banach $n$-ideal. ${\cal M} = {^a}\!{\cal M}$
if and only if $\cal M$ is closed and injective. Furthermore, ${^a}\!{\cal M}$ is the smallest closed
injective $n$-ideal containing ${\cal M}$.
\end{proposition}

\begin{proof} Assume that $\cal M$ is closed and injective. Given $A$
in ${^a}\!{\cal M}(E_1, \ldots, E_n;F)$, Theorem \ref{teorema} yields $i_F \circ A \in \overline{{\cal
M}(E_1, \ldots, E_n;\ell_\infty(B_{F^*}))}$. But $\cal M$ is closed, so \linebreak $i_F \circ A \in {\cal
M}(E_1, \ldots, E_n;\ell_\infty(B_{F^*}))$. From the injectivity of $\cal M$ it follows that $A \in {\cal
M}(E_1, \ldots, E_n;F)$, so ${\cal M} = {^a}\!{\cal M}$. The converse follows from the fact that ${^a}\!{\cal
M}$ is closed and injective (Proposition \ref{proposition}). Concerning the last assertion, we know that
${^a}\!{\cal M}$ is a closed injective $n$-ideal containing ${\cal M}$. Let ${\cal N}$ be a closed injective
$n$-ideal containing ${\cal M}$. Then, ${^a}\!{\cal M} \subseteq {^a}\!{\cal N} = {\cal N}$ by the first
assertion.
\end{proof}

We thus have that ${\cal M} \neq {^a}\!{\cal M}$ if $\cal M$ fails either to be closed or to be injective.
Concrete nonlinear examples will be given in both the closed non-injective case (Example
\ref{exemploprimeiro}) and the injective non-closed case (Example \ref{exemplonovo}).

\section{Finite type and compact mappings}

In the linear case, the closed injective hull of the ideal $\mathcal{F}$ of finite rank operators coincides
with the ideal ${\cal K}$ of compact operators, that is $^a \!{\cal F}=\mathcal{K}$ (see \cite[Proposition
19.2.3]{jarchow}). Denoting by ${\cal L}_{\cal K}$ the closed multi-ideal of compact multilinear mappings
(bounded sets are sent onto relatively compact sets), it is obvious to ask if the equality ${^a}\!({\cal
L}_f) = {\cal L}_{\cal K}$ holds true. We begin this section by giving a negative answer.

\begin{example}\rm Let $A\in {\cal L}(^2 \ell_2)$ be given by
$A((x_j)_{j=1}^\infty,(y_j)_{j=1}^\infty) = \sum_{j=1}^\infty x_jy_j$. Considering the canonical unit vectors
we see that $A$ is not weakly sequentially continuous, hence $A$ is not approximable. As $\overline{{\cal
L}_f}$ is closed, by Corollary \ref{corollary} we find that $A \notin {^a}\!(\overline{{\cal L}_f})(^2
\ell_2)$, hence $A \notin {^a}\!({{\cal L}_f})(^2 \ell_2)$. On the other hand, it is obvious that $A \in
{\cal L}_{\cal K}(^2 \ell_2)$.
\end{example}

Once we know that ${^a}\!({\cal L}_f) \neq {\cal L}_{\cal K}$, it is natural to look for another multi-ideal
which generalizes the compact operators and coincides with ${^a}\!({\cal L}_f)$. We will accomplish this task
almost entirely. First we need some terminology.

The notation $E_1, \stackrel{\hat i}{\ldots}, E_n$ means that
$E_i$ is omitted, that is: $(E_1, \stackrel{\hat i}{\ldots}, E_n)
= (E_1, \ldots, E_{i-1}, E_{i+1}, \ldots E_n)$, the same for
$(x_1, \stackrel{\hat i}{\ldots}, x_n)$. For $i = 1, \dots, n$,
consider the isometric isomorphism $I_i \colon {\cal L}(E_1,
\ldots, E_n;F) \longrightarrow {\cal L}{\Large(}E_i; {\cal L}(E_1,
\stackrel{\hat i}{\ldots}, E_n; F){\Large )}$,
\[I_i(A)(x_i)(x_1,  \stackrel{\hat i}{\ldots}, x_n) = A(x_1, \ldots, x_n).\]
For the case $n = 1$ to make sense, we consider $I_1(A) = A$ for $A \in {\cal L}(E;F)$. \\
\indent Given Banach operator ideals ${\cal I}_1, \ldots, {\cal I}_n$, an $n$-linear
mapping $A \in {\cal L}(E_1, \ldots, E_n;F)$ is said to be\\
$\bullet$ {\it of type $[{\cal I}_1, \ldots, {\cal I}_n]$}, and in this case we write $A \in [{\cal I}_1,
\ldots, {\cal I}_n](E_1, \ldots, E_n;F)$, if
\[I_i(A) \in {\cal I}_i{\Large(}E_i; {\cal L}(E_1,  \stackrel{\hat i}{\ldots}, E_n;F)
{\Large )}, {\rm~for~ every~} i = 1, \ldots, n.\] $\bullet$ {\it of type ${\cal L}({\cal I}_1, \ldots, {\cal
I}_n)$}, and in this case we write $A \in {\cal L}({\cal I}_1, \ldots, {\cal I}_n)(E_1, \ldots, E_n;F)$, if
there are Banach spaces $G_1, \ldots, G_n$, linear operators $u_j \in {\cal I}_j(E_j;G_j)$, $j = 1, \ldots,
n$, and $B \in {\cal L}(G_1, \ldots, G_n;F)$ such that $A
= B \circ (u_1, \ldots, u_n)$. If ${\cal I}_1 = \ldots = {\cal I}_n = {\cal I}$ we simply write
$[{\cal I}]$ and ${\cal L}({\cal I})$. \\
\indent It is well known that $[{\cal I}_1, \ldots, {\cal I}_n]$ and ${\cal L}({\cal I}_1, \ldots, {\cal
I}_n)$ are (closed, if ${\cal I}_1, \ldots,{\cal I}_n $ are closed) $n$-ideals (see \cite{note}). It is clear
that ${\cal L}({\cal I}_1, \ldots, {\cal I}_n) \subseteq [{\cal I}_1, \ldots, {\cal I}_n]$. If ${\cal I}_1,
\ldots, {\cal I}_n$ are closed and injective, then ${\cal L}({\cal I}_1, \ldots, {\cal I}_n) = [{\cal I}_1,
\ldots, {\cal I}_n]$ (\cite{bj,gg}). In particular, ${\cal L}({\cal K}) = [{\cal K}]$.

Our next aim is to show that ${^a}\!({\cal L}_f) = [\cal K]$ modulo the approximation property. Recall that
$[\cal K]$ coincides with the class of multilinear mappings which are weakly continuous on bounded sets
\cite{ahv}.

\begin{lemma}\label{lemma} Let ${\cal I}, \ldots, {\cal I}_n$ be operator ideals. If each ${\cal I}_j$
is injective,
then so are ${\cal L}({\cal I}_1, \ldots, {\cal I}_n)$ and $[{\cal I}_1, \ldots, {\cal I}_n]$.
\end{lemma}

\begin{proof} Let $A \in {\cal L}(E_1, \ldots, E_n;F)$ and
$i \colon F \longrightarrow G$ be an isometric embedding. If $i
\circ A \in {\cal L}({\cal I}_1, \ldots, {\cal I}_n)(E_1, \ldots,
E_n;G)$, we can find $G_1, \ldots, G_n$, $u_j \in {\cal
I}_j(E_j;G_j)$, $j = 1, \ldots, n$, and $B \in {\cal L}(G_1,
\ldots, G_n;G)$ such that $i \circ A = B \circ (u_1, \ldots,
u_n)$. For $j = 1, \ldots, n$, define $u_j^R \colon E_j
\longrightarrow \overline{{\rm Range}(u_j)}$ by $u_j^R(x_j) =
u_j(x_j)$, and let $i_j \colon \overline{{\rm Range}(u_j)}
\longrightarrow G_j$ be the formal inclusion. Thus 
$i_j \circ u_j^R = u_j$ belongs to ${\cal I}_j$. The injectivity
of ${\cal I}_j$ yields that each $u_j^R$ belongs to ${\cal I}_j$.
Define $C \in {\cal L}({\rm Range}(u_1), \ldots, {\rm
Range}(u_n);F)$ by $C(u_1(x), \ldots, u_n(x)) = A(x_1, \ldots,
x_n)$ and extend it continuously to a $\overline{ C} \in {\cal
L}(\overline{{\rm Range}(u_1)}, \ldots,\overline{{\rm
Range}(u_n)};F)$. So $A = \overline{ C} \circ (u_1^R, \ldots,
u_n^R)$, which shows that $A
\in {\cal L}({\cal I}_1, \ldots, {\cal I}_n)(E_1, \ldots, E_n;F)$.\\
\indent Suppose now that $i \circ A \in [{\cal I}_1, \ldots, {\cal I}_n](E_1, \ldots, E_n;G)$. Let $j \in
\{1, \ldots, n\}$. We know that $I_j(i \circ A) \in {\cal I}_j{\Large(}E_j; {\cal L}(E_1, \stackrel{\hat
j}{\ldots}, E_n;G){\Large )}$. Defining
$$J_j \colon {\cal L}(E_1,
\stackrel{\hat j}{\ldots}, E_n;F) \longrightarrow {\cal L}(E_1, \stackrel{\hat j}{\ldots}, E_n;G)~,~J_j(B) =
i \circ B,$$ it is clear that $J$ is an isometric embedding. For $x_j \in E_j$ and $(x_1, \stackrel{\hat
i}{\ldots}, x_n) \in E_1 \times \stackrel{\hat i}{\cdots} \times E_n,$ we have
\begin{eqnarray*}
[(J_j \circ I_j(A))(x_j)](x_1, \stackrel{\hat i}{\ldots}, x_n) & =
& [J_j ( I_j(A)(x_j))](x_1, \stackrel{\hat i}{\ldots}, x_n)\\
& =
& [i \circ ( I_j(A)(x_j))](x_1, \stackrel{\hat i}{\ldots}, x_n)\\
& =
& i ( ( I_j(A)(x_j))(x_1, \stackrel{\hat i}{\ldots}, x_n))\\
& =
& i ( A(x_1, \ldots, x_n))\\
& =
& (i \circ A)(x_1, \ldots, x_n)\\
& = & [(I_j(i \circ A))(x_j)](x_1, \stackrel{\hat i}{\ldots},
x_n).
\end{eqnarray*}
This shows that $J_j \circ I_j(A) = I_j(i \circ A)$. We have ${\cal I}_j$ injective, $J_j \circ I_j(A) \in
{\cal I}_j$ and $J_j$ is an isometric embedding. It follows that $I_j(A)$ belongs to ${\cal I}_j$, proving
that $A \in [{\cal I}_1, \ldots, {\cal I}_n](E_1, \ldots, E_n;F)$.
\end{proof}

\begin{proposition}\label{approximation} Suppose that $E_1^*, \ldots, E_n^*$ have the approximation
property. Then ${^a}\!({\cal L}_f)(E_1, \ldots, E_n;F)) = [{\cal K}](E_1, \ldots, E_n;F)$.
\end{proposition}

\begin{proof} It is well known that $[\cal K]$ is a closed (because
$\cal K$ is closed) multi-ideal. It is injective by Lemma \ref{lemma} as $\cal K$ is injective, hence $[{\cal
K}] = {^a}\![{\cal K}]$ by Proposition \ref{envoltoria}. From ${\cal L}_f \subseteq [{\cal K}]$ we conclude
that ${^a}\!({\cal L}_f) \subseteq {^a}\![{\cal K}] = [{\cal K}]$. As mentioned earlier, $[{\cal K}]$
coincides with the class of multilinear mappings which are weakly continuous on bounded sets \cite[Theorem
2.9]{ahv}. Supposing that $E_1^*, \ldots, E_n^*$ have the approximation property, \cite[Corollary 2.11]{ahv}
gives $[{\cal K}](E_1, \ldots, E_n;F) = \overline{{\cal L}_f}(E_1, \ldots, E_n;F)$. The proof is complete as
$\overline{{\cal L}_f}
 \subseteq {^a}\!({\cal L}_f)$.
\end{proof}

Next example shows that Proposition \ref{approximation} does not
hold true without the approximation property.

\begin{example}\label{exemplo}\rm Examining the proof of
\cite[Theorem 4.5]{acg} we find a Banach space $E$ lacking the approximation property and a compact symmetric
non-approximable linear operator $u \colon E \longrightarrow E^*$. Defining $A \in {\cal L}(^2E)$ by $A(x,y)
= u(x)(y)$ it is immediate that $A \in [{\cal K}](^2E)$ as $I_1(A) = I_2(A) = u$. Suppose that $A \in
\overline{{\cal L}_f}(^2E)$. Given $\varepsilon > 0$, there is $B \in {\cal L}(^2E)$ of finite type such that
$\|A - B\| < \varepsilon$. Say $B = \sum_{j=1}^k\varphi_i \psi_i$, $\varphi_1, \ldots, \varphi_k, \psi_1,
\ldots, \psi_k \in E^*$. Defining $v \in {\cal L}(E; E^*)$ by $v(x) = \sum_{j=1}^k\varphi_i(x) \psi_i$ we
have that $v$ is a finite rank operator. Furthermore,
$$\|u(x)(y) - v(x)(y)\| = \|A(x,y) - B(x,y)\| \leq \|A -
B\|\|x\|\|y\|< \varepsilon \|x\|\|y\| ,$$ for every $x,y \in E$. It results that $\|u - v\| \leq
\varepsilon$, which shows that $u$ is approximable, a contradiction. So, $A \notin \overline{{\cal
L}_f}(^2E)$, and by Corollary \ref{corollary} we have $A \notin {^a}\!({\cal L}_f)(^2E)$.
\end{example}

Thus far we know that ${^a}\!({\cal L}_f) = \overline{{\cal L}_f}$ if either the range space is injective
(Corollary \ref{corollary}) or the duals of the domain spaces have the approximation property (Proposition
\ref{approximation}). 
In the linear case, we have already mentioned that ${^a}\!{\cal F} = {\cal K}$, so any compact
non-approximable linear operator assures that ${^a}\!{\cal F} \neq \overline{\cal F}$. Let us see a nonlinear
example.

\begin{example}\label{exemploprimeiro}\rm Let $u \colon E \longrightarrow F$ be a compact
non-approximable linear operator, for example the operator from Example \ref{exemplo} (actually, for every
Banach space $E$ without the approximation property there is a Banach space $F$ and a compact
non-approximable operator $u$ from $E$ to $F$). Fix $\varphi \in E^*$, $\|\varphi\| = 1$, and $a \in E$ with
$\varphi(a) = 1$. Define
$$A \colon E \times E \longrightarrow F ~,~A(x,y) =
\varphi(x)u(y).$$ Suppose that $A$ belongs to $\overline{{\cal L}_f}$. Given $\varepsilon >0$, there are
$\varphi_1, \ldots, \varphi_k, \psi_1, \ldots, \psi_k \in E^*$ and $b, \ldots, b_k \in F$ such that $\|A -
\sum_{j=1}^k \varphi_j \psi_j b_j\| < \frac{\varepsilon}{\|a\|}$. For every $y \in E$ we have
$$\left \|u(y) - \sum_{j=1}^k \varphi_j(a) \psi_j(y) b_j \right
\| = \left \|A(a,y) - \sum_{j=1}^k \varphi_j(a) \psi_j(y) b_j
\right \|< \frac{\varepsilon}{\|a\|} \|a\|\|y\|,$$ resulting $\|u
- \sum_{j=1}^k \varphi_j(a) \psi_j b_j\| \leq \varepsilon$. But
$\sum_{j=1}^k \varphi_j(a) \psi_j b_j$ is a finite rank operator,
so $u$ is approximable, a contradiction. Hence $A \notin
\overline{{\cal L}_f}(^2E;F)$. In order to show that $A \in
{^a}\!({\cal L}_f)(^2E;F)$, let $\varepsilon >0$. Since $u$ is
compact, by \cite[Proposition 19.2.3]{jarchow} there is a finite
rank operator $v \colon E \longrightarrow \ell_\infty(B_{F^*})$
such that $\|i_F \circ u - v\| < \frac{\varepsilon}{\|\varphi\|}$.
Defining $B \in {\cal L}(^2E ;\ell_\infty(B_{F^*}))$ by $B(x,y) =
\varphi(x)v(y)$ we have that $B$ is of finite type and
\begin{eqnarray*} \|i_F \circ A(x,y) - B(x,y)\| & = & \|i_F( \varphi(x)u(y)) -
\varphi(x)v(y)\|\\
& =  &|\varphi(x)|\|i_F(u(y)) - v(y)\|\\
& \leq & \|\varphi\|\|i_F \circ u - v\|\|x\|\|y\|\\
& < & \varepsilon\|x\|\|y\|,
\end{eqnarray*}
for every $x,y \in E$. It follows that $\|i_F \circ A - B\| \leq \varepsilon$, so by Theorem \ref{teorema} we
have that $A \in {^a}\!({\cal L}_f)(^2E;F)$.
 \end{example}

We have seen that, contrary to the linear case, $[{\cal K}] \not\subseteq {^a}\!({\cal L}_f)$. Next we show
that this is caused by the fact that multilinear forms are not always of finite type. By ${\cal
L}_{\phi}(E_1, \ldots, E_n;F)$ we denote the subspace of ${\cal L}(E_1, \ldots, E_n;F)$ spanned by the
mappings of the form $A(x_1, \ldots, x_n) = B(x_1, \ldots, x_n)b$ where $B \in {\cal L}(E_1, \ldots, E_n)$
and $b \in F$.

\begin{proposition} If $A \in [{\cal K}](E_1, \ldots, E_n;F)$,
then for every $\varepsilon > 0$ there is an $n$-linear mapping $C \in {\cal L}_{\phi}(E_1, \ldots,
E_n;\ell_\infty(B_{F^*}))$ such that $\|i_F \circ A - C\| < \varepsilon$.
\end{proposition}

\begin{proof} Since ${\cal L}({\cal K}) = [{\cal K}]$, let $G_1,
\ldots, G_n$, $u_j \in {\cal K}(E_j;G_j)$ and $B \in {\cal L}(G_1,
\ldots, G_n;F)$ be such that $A = B \circ (u_1, \ldots, u_n)$. It
follows that $u_1 \otimes \cdots \otimes u_n$ is a compact
operator from $E_1 \widehat\otimes_\pi \cdots \widehat\otimes_\pi
E_n$ to $G_1 \widehat\otimes_\pi \cdots \widehat\otimes_\pi G_n$
(see \cite{k} 44.6.1). So, $i_F \circ B_L \circ u_1 \otimes \cdots
\otimes u_n$ is a compact operator from $E_1 \widehat\otimes_\pi
\cdots \widehat\otimes_\pi E_n$ to $\ell_\infty(B_{F^*})$. The
latter space has the approximation property, hence there are
$\varphi_1, \ldots, \varphi_k \in (E_1 \widehat\otimes_\pi \cdots
\widehat\otimes_\pi E_n)^* $ and $b_1, \ldots, b_k \in
\ell_\infty(B_{F^*})$ such that
$$\left \| i_F \circ B_L \circ u_1 \otimes \cdots \otimes u_n - \sum_{j=1}^k
 \varphi_jb_j \right\| < \varepsilon.$$
For $j = 1, \ldots, k$, take $B_j \in {\cal L}(E_1, \ldots, E_n)$
such that $(B_j)_L = \varphi_j$. For $(x_1, \ldots, x_n) \in  E_1
\times \cdots \times E_n$ it is clear that $B_L \circ u_1 \otimes
\cdots \otimes u_n(x_1 \otimes \cdots \otimes x_n) = B(u_1(x_1),
\ldots, u_n(x_n) = A(x_1, \ldots, x_n)$, so $\| i_F \circ A -
\sum_{j=1}^k  B_jb_j \| < \varepsilon$.
\end{proof}

\section{Weakly compact and absolutely summing mappings}

By ${\cal W}$ and ${\cal L}_{\cal W}$ we mean the ideals of weakly
compact linear operators and multilinear mappings respectively
(bounded sets are sent onto relatively weakly compact sets) and by
$\Pi_p$ the ideal of absolutely $p$-summing linear operators. In
this section we investigate multilinear counterparts of properties
of ${^a}\Pi_p$ and their connections with $\cal W$. The ideal
$\Pi_p$ has been generalized to the multilinear setting in several
ways, and among the most studied ones we find the multi-ideal of
dominated mappings:

\begin{definition}\rm Let $p_1, \ldots, p_n \geq 1$. An $n$-linear mapping $A \in {\cal L}(E_1, \ldots, E_n;F)$ is
$(p_1, \ldots, p_n)$-\textit{dominated} if there is a constant $C
\geq 0$ such that
$$\left( \sum_{j=1}^k\|A(x_j^1, \ldots, x_j^n)\|^r\right)^\frac{1}{r} \leq C \prod_{i=1}^n\sup_{\varphi
\in B_{E_i^*}}\left(\sum_{j=1}^k|\varphi(x_j^i)|^{p_i}
\right)^\frac{1}{p_i},$$ where $\frac{1}{r} = \frac{1}{p_1}+
\cdots +\frac{1}{p_n}$, for every $k \in \mathbb{N}$ and any
$x_j^i \in E_i$, $j= 1, \ldots, k$, $i =1, \ldots, n$. In this
case we write $A \in {\cal L}_{d;p_1, \ldots,p_n}(E_1, \ldots,
E_n;F)$. Denoting the infimum of the constants $C$ working in the
inequality by $\|A\|_{d;p_1, \ldots, p_n}$ we have that $({\cal
L}_{d;p_1, \ldots,p_n},\|\cdot\|_{d;p_1, \ldots, p_n})$ is a
complete $r$-normed $n$-ideal. If $p_1 = \cdots = p_n = p$ we say
that $A$ is $p$-dominated and write $A \in {\cal L}_{d;p}(E_1,
\ldots, E_n;F)$.
\end{definition}

%

Before going into the main results of the section we provide the
announced nonlinear examples of mappings in ${^a}\!{\cal M}$ but
not in $\cal M$ for injective non-closed $\cal M$. The
characterization
\begin{equation}{\cal L}_{d;p_1, \ldots,p_n} = {\cal L}(\Pi_{p_1}, \ldots,
\Pi_{p_n}),
\end{equation}
 which goes back to \cite{Pietsch}
(a detailed proof can be found in \cite[Corolario 3.23]{david}),
shall be useful several times.

\begin{example}\label{exemplonovo}\rm Since multilinear forms on $c_0$ are approximable we have that
${\cal L}(^nc_0)= {^a}({\cal L}_{d,p})(^nc_0)$ for every $p \geq
n$. On the other hand, ${\cal L}(^nc_0)\neq {\cal L}_{d,p}(^nc_0)$
for $n \geq 3$ and $p \geq 1$ by \cite[Theorem 3.5]{bote97}. So
${\cal L}_{d,p}(^nc_0) \neq {^a}({\cal L}_{d,p})(^nc_0)$ for $p
\geq n \geq 3$. Of course this is an implicit example. Let us see
an explicit one: let $u \in {\cal L}(E;F)$ be a $3$-summing
non-2-summing linear operator (for example the canonical map $j_3
\colon C[0,1] \longrightarrow L_3[0,1]$). Fix $\varphi \in E^*$,
$\|\varphi\| = 1$, and define $A \in {\cal L}(^2E;F)$ by $A(x,y) =
\varphi(x)u(y).$ Since $u$ is 3-summing, $u$ is absolutely
continuous, so by \cite[page 311]{djt} there are a Banach space
$G$ and a 2-summing operator $j \colon E \longrightarrow G$ such
that for every $\varepsilon > 0$ there is $K_{\varepsilon}$ such
that
$$\|u(x)\| \leq K_{\varepsilon}\|j(x)\| + \varepsilon \|x\| {\rm ~for~every~}x \in E.$$
Define $B \in {\cal L}(^2E;G)$ by $B(x,y) = \varphi(x)j(y)$. Since
$B = C \circ (\varphi, j)$, where $C(\lambda, y) = \lambda y$, it
follows from (*) that $B$ is 2-dominated. From
\begin{eqnarray*} \left \| \sum_{i=1}^k A(x_i,y_i)\right \| \!\!& =  &\!\! \left
 \| u \left(\sum_{i=1}^k \varphi(x_i)y_i \right)\right \| \leq K_{\varepsilon}
 \left\|j\left(\sum_{i=1}^k \varphi(x_i)y_i \right) \right\|
  + \varepsilon \left\|\sum_{i=1}^k \varphi(x_i)y_i \right\|\\
   \!\!& \leq  &\!\! K_{\varepsilon}\left
 \| \sum_{i=1}^k B(x_i,y_i) \right \|
  + \varepsilon \sum_{i=1}^k \|x_i\|\|y_i\|,
  \end{eqnarray*}
we conclude that $A \in {^a}({\cal L}_{d,2})(^2E;F)$. Suppose that
$A$ is 2-dominated. By (*) $A$ can be written as $A = C \circ
(v_1, v_2)$ with $v_1, v_2$ being 2-summing. Choosing $a \in E$
with $\varphi(a) = 1$, for every $x\in E$,
$$u(x) = \varphi(a)u(x) = A(a,x) = C(v_1(a), v_2(x)) = (C\circ(v_1(a),\cdot)\circ v_2)(x), $$
resulting $u = (C\circ(v_1(a),\cdot))\circ v_2$. This is absurd
because $v_2$ is 2-summing whereas $u$ is not, so $A$ fails to be
2-dominated.
\end{example}

Recall that ${^a}{\Pi_p} = {\cal H},$ the ideal of absolutely continuous operators.
 Since every absolutely continuous operator is
weakly compact and completely continuous \cite[Corollary 15.4]{djt}, it is natural to wonder whether every
multilinear mapping belonging to ${^a}\!{\cal L}_{d;p_1, \ldots, p_n}$ is (a) weakly compact and/or (b)
weakly sequentially continuous (for multilinear mappings the literature speaks of weakly sequentially
continuous mappings rather than completely continuous mappings).

Since there are $p$-dominated non-weakly compact $n$-linear
mappings \cite[Example 1]{jmaa}, we have $ {\cal L}_{d,p}
\subseteq {^a}\!{\cal L}_{d,p} \not\subseteq {\cal L}_{\cal W},$
so (a) does not hold in general. This leads us to consider a more
suitable standard multilinear generalization of weakly compact
operators: the ideal of Arens-regular multilinear mappings $[{\cal
W}].$

\begin{proposition} Every multilinear mapping belonging to ${^a}\!{\cal L}_{d;p_1, \ldots, p_n}$ is weakly
sequentially continuous and Arens-regular.
\end{proposition}

\begin{proof} 



First let us show that dominated multilinear mappings are weakly
sequentially continuous. Given $C \in {\cal L}_{d;p_1, \ldots,
p_n}(E_1, \ldots, E_n;F)$, applying (*) once more we can write $C
= B \circ (u_1,\ldots,u_n)$ where each $u_j$ is $p_j$-summing.
Since $p_j$-summing operators are completely continuous and $B$ is
continuous, it follows that $C$ is weakly sequentially continuous.
Now consider $A \in {^a}\!{\cal L}_{d;p_1, \ldots, p_n}(E_1,
\ldots, E_n;F)$. By Theorem \ref{teorema}, $i_F \circ A \in
\overline{{\cal L}_{d;p_1, \ldots, p_n}(E_1, \ldots,
E_n;\ell_\infty(B_{F^*}))}$. Moreover the space of weakly
sequentially continuous $n$-linear mappings from $E_1 \times
\cdots \times E_n$ to $\ell_\infty(B_{F^*})$ is closed and by the
first part of the proof it contains the $(p_1, \ldots,
p_n)$-dominated mappings, so $i_F \circ A$ is weakly sequentially
continuous. As $i_F$ is an isometric embedding, it follows that
$A$ is weakly sequentially continuous.

Recall that $p_j$-summing operators are weakly compact. Thus $
 [\Pi_{p_1},
\ldots, \Pi_{p_n}] \subseteq [{\cal W}].$
 Since $[{\cal W}] $ is
closed and injective and ${\cal L}(\Pi_{p_1}, \ldots, \Pi_{p_n}) \subseteq [\Pi_{p_1}, \ldots, \Pi_{p_n}],$
the second assertion follows from Proposition \ref{envoltoria}.
\end{proof}

\begin{remark}\rm Although there are Banach spaces $F$ such that $\mathcal{L}(\ell_1;F)={^a}{\Pi_p}(\ell_1;F),$ like
Hilbert spaces or reflexive quotients of a $C(K)$ space \cite[p. 313]{djt}, such type of coincidence does not
hold for multilinear mappings, that is $\mathcal{L}(^n\ell_1;F)\neq {^a}\! \mathcal{L}_{d;p_1, \ldots,
p_n}(^n\ell_1;F)$ for $n \geq 2$, $p_1, \ldots, p_n \geq 1$ and all Banach spaces $F$. Indeed, were the
equality true for some $n$, $p_1, \ldots, p_n$ and some $F,$ combining
 (*) with \cite[Proposition 41]{note} we would have $\mathcal{L}(^2\ell_1;F)= {^a}\! \mathcal{L}_{d;p_1,p_2}(^2\ell_1;F)$. So the same equality would hold true for any complemented subspace of $F$, thus for the scalar field.
However $\mathcal{L}(^2\ell_1) \neq {^a} \mathcal{L}_{d;p_1, p_2}(^2\ell_1)$ since the bilinear form on
$\ell_1$ constructed in \cite[p. 83]{acg1} is not Arens-regular.
\end{remark}

As mentioned earlier, ${^a}{\Pi_p} = {\cal H}$, so it follows that ${^a}{\Pi_p} = {^a}{\Pi_q}$ for every $p,q
\geq 1$ (see also \cite[Corollary 20.7.7]{jarchow}). Although we do not know whether ${^a}\!{\cal L}_{d;p_1,
\ldots, p_n} = {^a}\!{\cal L}_{d;q_1, \ldots, q_n}$, we are able to identify another multilinear
generalization of the ideal of absolutely summing linear operators in which this phenomenon does occur (as
usual, the properties of a given operator ideal are to be found among its several multilinear
generalizations, rather than in a specific one):

\begin{definition}\rm (Composition ideals - see \cite{bpr, FloretResults}) Let ${\cal I}$ be a Banach
operator ideal. An $n$-linear mapping $A \in {\cal L}(E_1, \ldots, E_n;F)$ belongs to ${\cal I} \circ {\cal
L}$ - in this case we write $A \in {\cal I} \circ {\cal L}(E_1, \ldots, E_n;F)$ - if there are a Banach space
$G$, an $n$-linear mapping $B \in {\cal L}(E_1, \ldots, E_n;G)$ and an operator $u \in {\cal I}(G;F)$ such
that $A = u \circ B$. $ {\cal I} \circ {\cal L}$ is a multi-ideal which becomes a Banach multi-ideal with the
norm
\[\|A\|_{{\cal I} \circ {\cal L}} := \inf\{\|u\|_{\cal I}\|B\|: A = u \circ B, B \in {\cal L}(E_1, \ldots,
E_n;G), u \in {\cal I}(G;F)\}.\]
\end{definition}

\begin{proposition} ${^a}\!({\cal I} \circ {\cal L}) = {^a}{\cal I} \circ {\cal L}$ for every
Banach operator ideal $\cal I$. In particular, ${^a}\!(\Pi_p \circ {\cal L}) = {^a}\!(\Pi_q \circ {\cal L})$
for every $p, q \geq 1$.
\end{proposition}

\begin{proof} Assume for a while that $\cal I$ is injective. Let $A \in {\cal L}
(E_1, \ldots, E_n;F)$ and $i \colon F \longrightarrow G$ be an isometric embedding such that $i \circ A \in
{\cal I} \circ {\cal L} (E_1, \ldots, E_n;G)$. 
Notice that according to \cite[Proposition 2.2]{bpr}, $(i \circ A)_L \in {\cal I}(E_1 \widehat\otimes_\pi
\cdots \widehat\otimes_\pi E_n;G)$. Since $i \circ A_L = (i \circ A)_L$, it follows that $i \circ A_L \in
{\cal I}(E_1 \widehat\otimes_\pi \cdots \widehat\otimes_\pi E_n;G)$. The injectivity of $\cal I$ gives $A_L
\in {\cal I}(E_1 \widehat\otimes_\pi \cdots \widehat\otimes_\pi E_n;F)$, so the factorization $A = A_L \circ
\sigma_n$, where $\sigma_n \colon E_1 \times \cdots \times E_n \longrightarrow E_1 \widehat\otimes_\pi \cdots
\widehat\otimes_\pi E_n$ is the canonical $n$-linear mapping given by $\sigma_n(x_1, \ldots, x_n)=x_1 \otimes
\cdots \otimes x_n$, shows that $A \in {\cal I} \circ {\cal L} (E_1, \ldots, E_n;F)$. Thus far we have proved
that ${\cal I} \circ {\cal L}$ is injective whenever $\cal I$ is injective. As ${^a}{\cal I}$ is closed and
injective, ${^a}{\cal I} \circ {\cal L}$ is closed and injective as well, so ${^a}{\cal I} \circ {\cal L} =
{^a}\!({^a}{\cal I} \circ {\cal L})$. From ${\cal I} \subseteq {^a}{\cal I}$ we get ${\cal I} \circ {\cal L}
\subseteq {^a}{\cal I} \circ {\cal L}$, hence ${^a}\!({\cal I} \circ {\cal
L}) \subseteq {^a}\!({^a}{\cal I} \circ {\cal L}) = {^a}{\cal I} \circ {\cal L}$.\\
\indent Given $A \in {^a}{\cal I} \circ {\cal L}(E_1, \ldots, E_n;F)$, write $A = u \circ B$ with $B \in
{\cal L}(E_1, \ldots, E_n;G)$ and $u \in {^a}{\cal I}(G;F)$. Given $\varepsilon
> 0$, there exists an operator $v \in {\cal I}(G;\ell_\infty(B_{F^*}))$ such that $\| v - i_F \circ u\| <
\frac{\varepsilon}{\|B\|}$. In this fashion $v \circ B \in {\cal I} \circ {\cal L}(E_1, \ldots,
E_n;\ell_\infty(B_{F^*}))$ and
\begin{eqnarray*}
\|v \circ B - i_F \circ A\|=\|v \circ B - i_F \circ u \circ B\|
\leq \| v - i_F \circ u\| \|B\| <
\varepsilon.
\end{eqnarray*}
This shows, by Theorem \ref{teorema}, that $A \in {^a}\!({\cal I} \circ {\cal L})(E_1, \ldots, E_n;F)$.
\end{proof}

\section{Aron-Berner extensions}
In this section we prove that the correspondence ${\cal M} \mapsto {^a}\!{\cal M}$ preserves Aron-Berner
stability. Recall that the Aron-Berner extension, sometimes called Arens extension or canonical extension, of
a multilinear mapping $A\in \mathcal{L}(E_1,\dots,E_n;F)$, denoted $\widetilde{A}\in
\mathcal{L}(E_1^{**},\dots,E_n^{**};F^{**})$, is defined  according to
$$ \widetilde{A}(x_1^{**},\ldots,x_n^{**})(\varphi)=\lim_{i_1} \cdots \lim_{i_n}(\varphi\circ
A)(x_{i_1},\dots,x_{i_n}) \;\;\text{ for all  } \varphi\in F^*,$$ where the limits are iterated limits and
$(x_{i_l})$ is a net in $E_l$ weak* converging to $x_l^{**}\in E_l^{**}.$ It turns out that $
\widetilde{A}(x_1^{**},\ldots,x_n^{**})=\lim_{i_1} \cdots \lim_{i_n} A(x_{i_1},\dots,x_{i_n})\;$ where the
limits are taken in the $w(F^{**},F^*)$ topology. The mapping  $$A\in \mathcal{L}(E_1,\dots,E_n;F)\mapsto
\widetilde{A}\in \mathcal{L}(E_1^{**},\dots,E_n^{**};F^{**})$$ is a linear into isometry. See \cite{din} for
more information about this bidual extension. Depending on the order the iterated limits are taken, $A$ can
have several (at most $n!$) different Aron-Berner extensions.

\begin{definition}\rm An $n$-ideal $\cal M$ is said to be {\it Aron-Berner stable} if all
Aron-Berner extensions of any $n$-linear mapping in $\cal M$ also
belong to $\cal M$.
\end{definition}

Some multi-ideals are known to be Aron-Berner stable (for example, integral mappings \cite[Proposition
8]{cl}) and some are known not to be (for example, weakly sequentially continuous mappings \cite[Example
1.9]{zalduendo}).

\begin{theorem}\label{ultimo} Let $\mathcal{M} $ be an Aron-Berner stable $n$-ideal. The following
are equivalent for a multilinear mapping $A \in {\cal L}(E_1, \ldots, E_n;F)$:\\
{\rm (a)} $A\in {^a}\!{\cal M}(E_1, \ldots, E_n;F)$.\\
{\rm (b)} All Aron-Berner extensions of $A$ belong to ${^a}\!{\cal M}(E_1^{**}, \ldots, E_n^{**};F^{**})$. \\
{\rm (c)} Some Aron-Berner extension of $A$ belongs to ${^a}\!{\cal M}(E_1^{**}, \ldots, E_n^{**};F^{**})$.\\
\indent In particular, ${^a}\!{\cal M}$ is Aron-Berner stable.
\end{theorem}

\begin{proof} (a) $\Longrightarrow$ (b) Let $\widetilde{A}$ be an Aron-Berner extension of $A$. In what
follows the notation $\widetilde{B}$ means the Aron-Berner extension of $B$ where the iterated limits are
taken in the same order as in $\widetilde{A}$. Firstly we remark that the natural embedding
 $$I_{F^{**}} \colon x^{**}\in F^{**}\mapsto (<x^{**},\varphi>)_{\varphi\in B_{F^*}}\in
 \ell_\infty(B_{F^*}),$$is a linear isometry that extends the isometric embedding $i_F$
 and it is weak*-weak* continuous since it is the transpose of the
 mapping $$(a_\varphi)_{\varphi\in B_{F^*}}\in\ell_1(B_{F^*}) \mapsto \sum_{\varphi\in B_{F^*}} a_\varphi
 \cdot
 \varphi
 \in F^*.$$ This remark proves that $\widetilde{i_F\circ A}= I_{F^{**}}\circ\widetilde{ A}.$
Assume that $A\in {^a}\!{\cal M}(E_1,\dots,E_n;F)$ and fix $\varepsilon >0.$ According to condition (b) in
Theorem \ref{teorema}, there is an $n$-linear mapping $C\in \mathcal{M}(E_1,\dots,E_n;\ell_\infty(B_{F^*}))$
such that $\|i_F\circ A - C\| < \varepsilon.$ Hence
$$\|I_{F^{**}}\circ\widetilde{
A}-\widetilde{C}\|=\|\widetilde{i_F\circ A} -
\widetilde{C}\|=\|\widetilde{(i_F\circ A-C)}\|< \varepsilon.$$ The
Aron-Berner stability of $\cal M$ assures that $\widetilde{C}$
belongs to $\cal M$, thus we have seen that
$I_{F^{**}}\circ\widetilde{A}$ is in the closure of
$\mathcal{M}(E_1^{**},\dots,E_n^{**};(\ell_\infty(B_{F^*}))^{**}).$
Hence $I_{F^{**}}\circ\widetilde{A} \in {^a}\!{\cal
M}(E_1^{**},\dots,E_n^{**};(\ell_\infty(B_{F^*}))^{**}),$ and
because of the injectivity of ${^a}\!{\cal M},$
$\widetilde{A} \in {^a}\!{\cal M}(E_1^{**},\dots,E_n^{**};F^{**})).$ \\
(b) $\Longrightarrow$ (c) is obvious.\\
(c) $\Longrightarrow$ (a) By $J_E$ we mean the canonical isometric embedding from $E$ into $E^{**}$. Let
$\widetilde A$ be an Aron-Berner extension of $A$ belonging to ${^a}\!{\cal M}(E_1^{**}, \ldots,
E_n^{**};F^{**})$. The ideal property yields that $J_F \circ A = \widetilde{A} \circ (J_{E_1}, \ldots,
J_{E_n}) \in {^a}\!{\cal M}(E_1, \ldots, E_n;F^{**})$. From the injectivity of ${^a}\!{\cal M}$ it follows
that $A  \in {^a}\!{\cal M}(E_1, \ldots, E_n;F)$.
 \end{proof}

It is well known that a linear operator $u$ belongs to ${^a}\Pi_p$, that is, $u$ is absolutely continuous, if
and only
 if $u^{**}$ belongs to ${^a}\Pi_p$ as well \cite[Corollary 15.5]{djt}. It is clear that Theorem \ref{ultimo}
generalizes this result to multilinear mapings. Moreover, taking
$n=1$ in Theorem \ref{ultimo} one sees that even in the linear
case
 \cite[Corollary 15.5]{djt} is a particular case of a much more general situation:

\begin{corollary} Let $\cal I$ be an operator ideal such that $u \in {\cal I} \Longrightarrow u^{**} \in {\cal I}$.
Given $u \in {\cal L}(E;F)$, $u \in {^a}{\cal I}(E;F)$ if and only if $u^{**} \in {^a}{\cal
I}(E^{**};F^{**})$.
\end{corollary}

For instance, maximal Banach operator ideals satisfy the condition
$u \in {\cal I} \Longrightarrow u^{**} \in {\cal I}$
\cite[Corollary 17.8.4]{df} (observe that, for being closed,
${^a}{\cal I}$ is maximal only if ${^a}{\cal I}= {\cal L}$). Let
us stress that the above corollary is a genuine generalization of
\cite[Corollary 15.5]{djt}: indeed, on the one hand the proof of
\cite[Corollary 15.5]{djt} works only for operator ideals
contained in $\cal W$, on the other hand there are maximal
operator ideals not contained in $\cal W$, for example the ideal
of cotype 2 operators (cf. \cite[17.4]{df}).

\medskip

\noindent {\sc Acknowledgements.} Part of this paper was written
while G.B. and L.P. were visiting the Departamento de An\'alisis
Matem\'atico at Universidad de Valencia. They thank Pablo Galindo,
Pilar Rueda and the members of the department for their
hospitality and support.

\small{
\vspace*{1em}
 \noindent [Geraldo Botelho] Faculdade de Matem\'atica, Universidade Federal de Uberl\^andia,
38.400-902 - Uberl\^andia, Brazil,  e-mail: botelho@ufu.br.

\medskip

\noindent [Pablo Galindo] Departamento de An\'alisis Matem\'atico,
Universidad de Valencia, 46.100 Burjasot - Valencia, Spain,
e-mail: Pablo.Galindo@uv.es.

\medskip

\noindent [Leonardo Pellegrini] Departamento de Matem\'atica,
Instituto de Matem\'atica e Estat\'istica, Universidade de S\~ao
Paulo, Caixa Postal 66281, 05315-970 S\~ao Paulo, Brazil, e-mail:
leonardo@ime.usp.br.}
\end{document}